\documentclass[12pt]{amsart}
\usepackage{amsmath,amsthm,amscd,amsfonts,amssymb}
\usepackage[all]{xy}

\begin{document}
   \baselineskip   .7cm
\newtheorem{defn}{Definition}[section]
\newtheorem{thm}{Theorem}
\newtheorem{propo}[defn]{Proposition}
\newtheorem{cor}[defn]{Corollary}
\newtheorem{lem}[defn]{Lemma}
\theoremstyle{remark}
\newtheorem{rem}{Remark}[section]
 \newtheorem{ex}[defn]{Example}
\renewcommand\o{{{\mathcal O}}}
\newcommand\s{\sigma}
\newcommand\w{\widehat}
\newcommand\cal{\mathcal}

\title [A calculation of $L$-series in terms of Jacobi sums]
{A calculation of $L$-series in terms of Jacobi sums}

\author[A. \'Alvarez]{A. \'Alvarez${}^*$     }
\address{Departamento de Matem\'aticas  \\
Universidad de Salamanca \\ Plaza de la Merced 1-4. Salamanca
(37008). Spain.}

\thanks{$*$  Departamento de Matem\'aticas.
Universidad de Salamanca. Spain.}

\maketitle

  \begin{abstract}  Let us consider a cyclic extension of a function field   defined  over a finite field. For a character (non-trivial) of this extension, we calculate, as a linear combinations of products of Jacobi sums, the coefficients of the polynomial given by its Dirichtlet $L$-series.
 In the last section we show applications of this calculation.

 \end{abstract}

\tableofcontents


\section{Introduction  }

In \cite{W} it is proved that the Hasse-Weil zeta function for   Fermat curves     has an analytic continuation to the whole complex plane.  It is possible because   these curves are   abelian  coverings  of the projective line, ramified at three points, and hence  it is deduced that the $p$-local terms  of theirs Hasse-Weil  zeta functions are  $L$-series  associated with  a  Hecke characters  which came from a Jacobi sums.

If we have cyclic ramified coverings of a smooth curve, defined over finite fields, then the $L$-series associated with the characters (non trivial) for these coverings are polynomials.
In \cite
{D}, it is proved that the constant term  of these polynomials is given by Jacobi sums.
By using an elementary method, in \cite{A}  another proof of this result can be obtained.

In this article we calculate, in terms of Jacobi sums, the coefficients of these polynomials. We shall show that these coefficients are explicitly  given by linear combinations of products of Jacobi sums.

In the last section, we provide some applications of this calculation. Thus, we find formulas similar   to the Deuring polynomial and those of \cite{B} in Chapter 5 but in  characteristic $0$. Moreover, in characteristic $p$ these formulas are more explicit. This is because  Jacobi sums are   combinatorial numbers in  mod $p$. From these calculations, we associate with each finite field ${\Bbb F}_q$
a Hermitian form such that its rank and its index determine the amount of $\lambda \in {\Bbb F}_q$, for which elliptic curves $y^2=x(x-1)(x-\lambda)$ are supersingular and the cardinal of theirs rational points are greater than $q+1$, respectively.

\section{Previous notation and preliminaries }
Let us consider the primitive $n$-root of the unity  $\epsilon_n$ and a prime ideal,  ${\mathfrak p}\subset {\Bbb Z}[\epsilon_n]$,   over the  prime    $p$, prime to $n$, with residual field
${\Bbb F}_q$,  a finite field of  $q=p^h$ elements. Note that in this case $n$ divides $q-1$.

Let $Y,X$ be   smooth, proper and geometrically irreducicle curves over the finite field   ${\Bbb F}_q$ and let $\Sigma_{Y } $, $\Sigma_{X } $ be the function fields of $Y $ and $X $, respectively. Now,  let us  consider $Y \to X $, a  ramified Galois covering with Galois group $G:={\Bbb Z}/n$,  ramified at $(d+1)$-different rational points of $X$, $T:=\{  {x}_0,\cdots,{  x}_d\}$. We have that $\Sigma_{Y }=\Sigma_{X }(\sqrt[n]{ f })$. Let ${\rm{div}}(f )$ be the principal divisor associated with $f \in \Sigma_{X }$. We can choose $f$ such that
${\rm{div}}(f)=n_0\cdot x_0+\cdots+n_d\cdot x_d+n\cdot D$, with $0<  n_i  <n$, and where $D$  is a divisor on $X$  (note that $n_0  +\cdots+n_d \equiv 0\,{\rm{mod}}\, n$). We denote by $k(x_i)$ the residual field of $x_i$.

 We consider the effective divisors  $E:=  x_0+x_1+\cdots+  x_d$, $E':=  x_1+\cdots+  x_d$  and by $\mathfrak m$, $\mathfrak m'$ the ideals in  ${\mathcal O}_{X }$ associated with $E$ and $E'$, respectively. We also denote $T':=\{  {x}_1,\cdots,{  x}_d\}$. An $\mathfrak m$-level structure over $X$ is a pair $(L,\iota_\mathfrak m)$, where $L$ is a line bundle over $X$ and $\iota_\mathfrak m:L\to {\mathcal O}_{X}/{\mathfrak m}$ is an surjective morphism of ${\mathcal O}_{X }$-modules. We say that two level structures $(L,\iota_\mathfrak m)$ and $(L',\iota'_\mathfrak m)$ are isomorphic  when there exists an isomorphism of ${\mathcal O}_X$-modules $ \tau:L\to L'$ such that $\iota'_\mathfrak m\cdot \tau =\iota_\mathfrak m$.

 It is not hard to see that   the tensor product defines a group law for the level structures. We denote by ${\rm{Pic}}^0_{X,\mathfrak m  }({\Bbb F}_q)$ the group of level structures $(L,\iota_\mathfrak m)$, where ${\rm{deg}}(L)=0$. Moreover, by considering the morphism of forgetting  $\iota_\mathfrak m$, $\theta(L,\iota_\mathfrak m)= L$, we have that
 $${\rm{Ker}}(\theta)=({\mathcal O}_X/{\mathfrak m})^\times/{\Bbb F}_q^\times$$
  and hence,
we obtain an exact sequence of groups
 $$1\to k(  x_0)^\times \times\cdots\times k(  x_d)^\times/{\Bbb F}_q^\times\overset \eta \to {\rm{Pic}}^0_{X,{\mathfrak m} }({\Bbb F}_q)\overset \theta\to   {\rm{Pic}}^0_{X }({\Bbb F}_q) \to 1.$$

 By class field theory, \cite{S} VI, since $   \Sigma_{X }(\sqrt[n]{ f })/\Sigma_{X }$ is a cyclic extension of group $G$, ramified at $T$, there exists a surjective morphism of groups $\rho:{\rm{Pic}}^0_{X_{\mathfrak m} }({\Bbb F}_{q})\to G$. If $g\in {\rm{Pic}}^0_{X,{\mathfrak m} }({\Bbb F}_{q})$, then $\rho(g) (\sqrt[n]{ f })=\lambda\cdot (\sqrt[n]{ f })$ with $\lambda \in {\Bbb F}_{q}^\times$. We define the character $\chi_f$ of $ {\rm{Pic}}^0_{X,{\mathfrak m} }({\Bbb F}_{q})$ by $\chi_{f}(g)=:\chi_{q}^{\frac{q-1}{n}}(\lambda)$, $\chi_{q}$ being the Teichmuller character:
if $\lambda \in {\Bbb F}_{q}$, then we set the $(q-1)$-root of the unity, $\chi_{q} (\lambda)\in {\Bbb Z}[\epsilon_{q-1}]$, by the condition
 $\chi_{q} (\lambda)\equiv\lambda{\text{ mod } }\mathfrak p'$, by fixing $\mathfrak p'\subset {\Bbb Z}[\epsilon_{q-1}] $,  a prime over $\mathfrak p$. Note that  ${\Bbb Z}[\epsilon_n]\subset{\Bbb Z}[\epsilon_{q-1}]$.
 We have that $\chi_{q}$ does not depend on the choice of  $\mathfrak p'$.

 If we consider the injective morphism
 $$\eta: k(  x_0)^\times \times\cdots\times k(  x_d)^\times/{\Bbb F}_q^\times\to {\rm{Pic}}^0_{X,{\mathfrak m} }({\Bbb F}_q)$$
 then, by denoting $a_i:=n_i \frac{q-1}{n}$, $(0\leq i\leq d)$, we have that
 $\chi_f(\eta(z_0,\cdots,z_d))=\chi^{- a_0}_{q}(z_0)\cdots \chi^{- a_d}_{q}(z_d). $

Now,  we shall build a section for the morphism $\theta:{\rm{Pic}}^0_{X,{\mathfrak m} }({\Bbb F}_q) \to  {\rm{Pic}}^0_{X }({\Bbb F}_q)$:
 Let $\pi_{\mathfrak m}: {\mathcal O}_{X }\to {\mathcal O}_{X }/{\mathfrak m}$ be the natural epimorphism.
Let  $D$ be a divisor on $X$. We shall fix an ${\mathfrak m}$-level structure  for the line bundle  ${\mathcal O}_{X }(D)$.
 If $D$  is an effective divisor on $X$   with support outside $T$    then by considering the natural morphisms ${\mathcal O}_{X  }(-D)\to {\mathcal O}_{X  }$, ${\mathcal O}_{X  }\to {\mathcal O}_{X  }( D)$ and the level structure $({\mathcal O}_{X },\pi_{\mathfrak m})$  we obtain level structures   $ ({\mathcal O}_{X  }(\pm D),\pi_{\mathfrak m}^{\pm D}) $. Similarly, if $D'=D_1'-D_2'$ is a divisor with support away of $T$ and $D_1$, $D_2$ ares effective divisors,  then $({\mathcal O}_{X  }(D_1'),\pi^{  D_1'})\otimes ({\mathcal O}_{X  }(-D_2'), \pi_{\mathfrak m}^{  -D_2'})$ gives a  level structure for $ {\mathcal O}_{X  }(D )$.

 Let $\mathfrak m_{x_0}$ be the ideal associated with $x_0$. We now consider $  D  =r\cdot x_0+  K $ with     $x_0\notin \rm{supp}(K) $ and we set $t_{x_0}$ a local parameter for $x_0$. As above, we obtain an $x_0$-level structure, $\pi^{\overline D}_{\mathfrak m_{x_0} }$, for $t_{x_0}^{-r}\cdot {\mathcal O}_X(D )\simeq {\mathcal O}_X(\overline D )$ since that $ x_0\notin\rm{supp}(\overline D)$. Thus, from the isomorphisms of
 ${\mathcal O}_{X}$-modules
 ${\mathcal O}_{X}(D )\overset{t_{x_0}^{-r}}\simeq t_{x_0}^{-r}\cdot {\mathcal O}_{X}(D)$ and from the $\mathfrak m_{x_0}$-level structure for ${\mathcal O}_{X}(\overline D )$,   we obtain an epimorphism   $\pi_{\mathfrak m_{x_0}}^{  D}:{\mathcal O}_{X}(D)\to {\mathcal O}_{X  }/\mathfrak m_{x_0}$. For each  divisor $F$ on $X\setminus T'$ we denote $l^F_{\mathfrak m_{x_0}}:=(\mathcal O_{X}(F),\pi^F_{\mathfrak m_{x_0}}\times \pi^F_{\mathfrak m'})$.

{\bf{Example 1.}} If we consider ${X}={\Bbb P}_1$, with $\Sigma_{{\Bbb P}_1}={\Bbb F}_q(x)$, ${\rm{  div}}(x  )=0-\infty$ and  the local parameter $x^{-1}$ for $\infty$, then the $\infty$-level structure $l_\infty^{r\cdot \infty}  =({\mathcal O}_{{\Bbb P}_1}(r\cdot \infty),\pi_\infty^{r\cdot \infty})$  is given  by
 $\pi_{ \infty}^{r\cdot \infty}:{\mathcal O}_{{\Bbb P}_1}(r\cdot \infty)\to {\mathcal O}_{{\Bbb P}_1}/ \mathfrak m_\infty$,
    with $\pi_\infty^{r\cdot \infty}(x^r)=1$.  Recall that $H^0({\Bbb P}_1 ,{\mathcal O}_{{\Bbb P}_1}(r\cdot \infty ) ) ={\Bbb F}_q\cdot 1\oplus {\Bbb F}_q\cdot x  \oplus \cdots  \oplus{\Bbb F}_q\cdot x^r$.
\bigskip

Let $L$ be a line bundle on $X $. By taking account the identification
 ${\rm{Hom}_{{\mathcal O}_{X  }}}({\mathcal O}_{X  }, L)=H^0(X ,L ) $,   we can consider the global sections of $L$ as ${\mathcal O}_{X  } $-module morphisms ${\mathcal O}_{X  }\to L$.
We define the   space of $\mathfrak m$-sections of a level structure, $(L,\iota_\mathfrak m)$,  as the subset   $H^0_\mathfrak m(X,(L,\iota_\mathfrak m))\subset H^0(X ,L ) $,    of global sections of $L$, $s:{\mathcal O}_{X  }\to L$ such that $\iota_{\mathfrak m}\cdot s=\pi_\mathfrak m$.  We have that if $s,s'\in H^0_\mathfrak m(X,(L,\iota_\mathfrak m))$ then $s-s'\in  H^0(X ,L(-E) ) $. Thus, $H^0_\mathfrak m(X,(L,\iota_\mathfrak m))=s+H^0(X ,L(-E) )$. We denote $ h^0_\mathfrak m (L,\iota_\mathfrak m ):=\#H^0_\mathfrak m(X,(L,\iota_\mathfrak m))$. Note that either $h^0_\mathfrak m (L,\iota_\mathfrak m )=0$ or $h^0_\mathfrak m(L,\iota_\mathfrak m )=q^{h^0(L(-E) )}  $.

 \section{L-series}

Let $F_x$ be the Frobenius element for $x\in X  \setminus T$. We consider the $L$-series associated with the character $\chi_f$,
$$L(X,\chi_f,t)=\underset{x\in \vert X \vert}\prod (1-\chi_f(F_x)\cdot t^{deg(x)})^{-1},$$
$\vert X \vert$ denotes the geometric point within $X$. If $x\in T$ then we set $\chi_f(F_x)=0$.

We consider a divisor of degree $1$, $D_1$ on $X$ and we denote $L(i):=L\otimes_{{\mathcal O}_X}{\mathcal O}_X(i\cdot D_1)$. In \cite[3]{Al}, by  using    \cite[4.1.1]{A},  this $L$-series is calculated in terms of the $\mathfrak m$-level structures

 $$ (*) L(X,\chi_f,t)=   \underset{i=0}{\overset
{ 2g+d-1}  \sum}  [\sum_{(L,\iota_\mathfrak m)}     h^0_{\mathfrak m} (L(i),\iota_{\mathfrak m}) \cdot\chi_f( {\rho(L,\iota_{\mathfrak m})})]\cdot t^{i}  .$$
The second sum is over all level structures classes $(L,\iota_\mathfrak m)\in {\rm{Pic}}^0_{X,{\mathfrak m} }({\Bbb F}_q)$.

Note that the zeta function of $Y$ is
$$Z(Y, t)= [(1-t)(1-q\cdot t)]^{-1} \prod_{j=1}^{n-1}L(t,\chi^j_f) .$$

If $\lambda \in {\Bbb F}_q^\times$ and $(L ,\iota_\mathfrak m)$ is a $\mathfrak m$-level structure then $(L ,\iota_\mathfrak m)$ and $(L ,\lambda\cdot \iota_\mathfrak m)$ are isomorphic level structures. Moreover, as  ${\mathcal O}_{X  }/\mathfrak m\simeq k(x_0)\times {\mathcal O}_{X  }/\mathfrak m'$  we have that $\iota_\mathfrak m   =\iota_{x_0}\times \iota_{\mathfrak m'} $
 with $\iota_{x_0}:{\mathcal O}_{X  }\to  k(x_0)$ and $\iota_{\mathfrak m'}:{\mathcal O}_{X  }\to {\mathcal O}_{X  }/\mathfrak m'$   epimorphisms. Thus, in the isomorphism class of a level structure  $(L ,\iota_\mathfrak m)$   we can choose the element $(L ,\pi^D_{x_0}\times \iota_{\mathfrak m'})$, with $L \simeq {\mathcal O}_{X  }(D)$, and therefore  we can fix the $x_0$-level structure in the classes of ${\mathfrak m}$-level structures. Let $F$ be either $E$ or $E'$.
We denote $O_{F}:=H^0(X,\frac{\mathcal O_X}{\mathcal O_X(-F)}) $.
Since we fix  the $x_0$-level structure, the subgroup $1\times {  O}_{E'}^\times \subset {  O}_{E}^\times$ acts transitively on the classes of $\mathfrak m$-level structures of the fiber  $\theta^{-1}(L )$. Recall that
$\theta:{\rm{Pic}}^0_{X,{\mathfrak m} }({\Bbb F}_q) \to   {\rm{Pic}}^0_{X }({\Bbb F}_q)$  is the morphism of forgetting level structures.

Let $J$ be a divisor on $X\setminus T'$. We consider the level structure $ l^J_\mathfrak m :=(\mathcal O_X(J),\pi^J_\mathfrak m)$,  fixed in the above section. We denote    $$H^J_\mathfrak m:=\pi^{J}_\mathfrak m(\frac{H^0(X,{\mathcal O}_{X  }(J))}{H^0(X,{\mathcal O}_{X  }(J-E) )})\subseteq {  O}_E.$$

 \begin{propo}\label{series} We have
 $$ L(X,\chi_f,t)=   \underset{i=0}{\overset
{ 2g+d-1}  \sum}  [\sum_{[D]}      q^{h^0  ({\mathcal O}_{X  }(D+i\cdot D_1-E))} \cdot\chi_f( {\rho( l^D_{\mathfrak m})})\sum_{u }\chi_f(\eta(1\times u ))]\cdot t^{ i}   .$$
The second sum is over all the classes (with the algebraic equivalence) of divisors of degree $0$, $D$,     on $X\setminus T'$ and the third sum is over all $u\in {  O}^\times_{E'}$ with $1\times u^{-1}\in  H_\mathfrak m^{D+i\cdot D_1} \subset O_E$.
 \end{propo}
 \begin{proof}By considering the  level structure $l^{D+iD_1}_\mathfrak m$    and bearing in mind that if
 $({\mathcal O}_{X  }(D)(i ), \pi_{x_0}^{D+iD_1} \times \iota_{\mathfrak m'})$ is another level structure  for  ${\mathcal O}_{X  }(D)(i )$, then there exists $u\in {  O}^\times_{E'}$ with $  \pi_{x_0}^{D+iD_1} \times \iota_{\mathfrak m'}  = (1\times u)\cdot \pi^{D+iD_1}_\mathfrak m$. We have that the $L$-series (*)  is equal to

  $$\underset{i=0}{\overset
{ 2g+d-1}  \sum}  [\sum_{[D]}      q^{h^0  ({\mathcal O}_{X  }(D+i\cdot D_1-E))} \cdot\chi_f( \rho(l^{D }_\mathfrak m))\sum_{u}\chi_f(\eta(1\times u ))]\cdot t^{ i}   ,$$
 where the third sum is over all   $u\in {  O}^\times_{E'}$ with $h^0_\mathfrak m(  (1\times u)\cdot l^{D+iD_1}_\mathfrak m)\neq 0$, and in these cases, $h^0_\mathfrak m(  (1\times u)\cdot l^{D+iD_1}_\mathfrak m)=
 q^{h^0( {\mathcal O}_{X  }(D+i\cdot D_1-E)   )}$.
   Note  that we have fixed the $x_0$-level structure  in the classes of $\mathfrak m$-level structures.

 To conclude, it suffices to prove that $h^0_\mathfrak m(  (1\times u)\cdot l^{D+iD_1}_\mathfrak m)\neq 0$   if and only if $1\times u^{-1}\in  H_\mathfrak m^{D+i\cdot D_1} $.
 If $h^0_\mathfrak m((1\times u)\cdot l^{D+iD_1}_\mathfrak m)\neq 0$ then there exists $s\in  H^0(X,{\mathcal O}_{X  }(D)(i))$ such that the ${\mathcal O}_X$-module morphism $s:{\mathcal O}_X\to {\mathcal O}_{X  }(D)(i)$ satisfies  $(1\times u)\cdot \pi_\mathfrak m^{D+iD_1} \cdot s=\pi_\mathfrak m$. Therefore, $(1\times u)\cdot\pi_\mathfrak m^{D+iD_1}(s)=1$. Conversely, if $1\times u^{-1}\in  H_\mathfrak m^{D+i\cdot D_1} $ there exists  $s\in  H^0(X,{\mathcal O}_{X  }(D)(i))$
 with $1\times u^{-1}=\pi^{D+iD_1}_\mathfrak m(s)$. Thus, $s\in  H^0_\mathfrak m( X, (1\times u)\cdot l^{D+iD_1}_\mathfrak m)$ and we conclude.
 \end{proof}

   {\bf{Example 2}}. As in Example 1, we consider $X={\Bbb P}_1$, the line bundle $ \mathcal O_{{\Bbb P}_1}(r\cdot \infty)$, $E:=\infty+x_1+\cdots+x_d$, with $ r\leq d-1 $, and   $f$ the polynomial of degree $d$, $p(x)$, where $x_1,\cdots,x_d$ are given by the roots  of this polynomial. We consider the $\mathfrak m $-level structure $l^{r\cdot \infty}_\mathfrak m$;  recall that over $\mathfrak m'$ this level structure is obtained from the natural injective morphism $\mathcal O_{{\Bbb P}_1} \to \mathcal O_{{\Bbb P}_1}(r\cdot \infty)$ and over $\infty$ it is given by $ \pi_\infty^{r\cdot \infty} $ by considering $\pi_\infty^{r\cdot \infty}(x^r)=1$(Example 1).
In this way, the subspace
   $$H^{r\cdot \infty}_\mathfrak m:=\pi^{r\cdot \infty}_\mathfrak m(\frac{H^0({\Bbb P}_1,\mathcal O_{{\Bbb P}_1}(r\cdot \infty))}{H^0({\Bbb P}_1,\mathcal O_{{\Bbb P}_1}(r\cdot \infty-E))})\subseteq O_E={\Bbb F}_q \times \frac{{\Bbb F}_q[x]}{p(x)}  $$
    is $\{((\frac{h(x)}{x^r})_{x=\infty},h(x)){\text{ with }} {\rm{ deg}}(h(x))\leq r \}.$
Therefore, $1\times u(x) \in  H^{r\cdot \infty}_\mathfrak m$ if and only if $deg(u(x) )=r$ and $u(x)$ is monic.


\section{Jacobi sums and L-series}

Let us consider the linear form   $\omega_d(z)=\gamma_1\cdot  z_1+\cdots +\gamma_d \cdot z_d+\gamma$, where $\gamma_1,\cdots,\gamma_d, \gamma\in {\Bbb F}_q$ and    $z:=(z_1,\cdots,z_d)\in {\Bbb F}^d_q $ with $d\geq 2$. We define the Jacobi sum with    $(a)=(a_1, \cdots,a_d)\in  {\Bbb Z}^d $,  by
$$J_{\omega_d}^{(a)} :=
 \underset { z , \, \omega_d(z)=0   }\sum \chi_q^{a_1}(z_1)\cdots \chi_q^{a_d}(z_d)\in {\Bbb Z}[\epsilon_{q-1}].$$
When $z_i=0$, we set $\chi^{a_i}_q (z_i)=0$. The definition of the Teichmuller character, $\chi_q$,  is given in section 2.

Let $H^d_r\subset {\Bbb F}^d_q$  be an affine subvariety  of  dimension $r$. We define
$$J_{H^d_r }^{(a)}:=
 \underset { z\in H^d_r     }\sum \chi_q^{a_1}(z_1)\cdots \chi_q^{a_d}(z_d).$$

 We also define $J^{(a)}=(-1)^{d-1}\sum_{z_1+\cdots+z_d+1=0}\chi_{q}^{a_1 }(z_1)\cdots\chi_{q}^{a_d}(z_d)$.

Note, that   $J_{H^d_r }^{(a)}$ is defined for $a_i$ mod $(q-1)$, thus, we can consider $\vert a_i \vert\leq  q-1$.

 We can assume, by reordering the entries if   necessary, that the $r$-dimensional affine variety $H_r^d\subset {\Bbb F}_q^d$ is given by the equations
 $$\{\omega_{r+1}(z)=0,\cdots, \omega_{d}(z)=0\}$$ with $\omega_k(z)= -\sum^r_{j=1}\alpha_{kj}z_j+z_k-\alpha_k, $
  and $r+1\leq k\leq d. $

\begin{lem}\label{car} Let us consider $(z_1,\cdots,z_r)\in {\Bbb F}^r _q$, with $z_i\neq 0$. We have the equality
$$ \chi_q^{a_k}( \alpha_{k1}\cdot  z_1+\cdots + \alpha_{kr } \cdot z_{r }+ \alpha_k )=$$
$$  (q-1)^{-r}  \sum_{0\leq i_1,\cdots,i_{r} \leq q-2}  \chi_{q}^{ i_1}(z_1)\cdots \chi_{q}^{ i_{r}}(z_{r})\cdot J_{\omega_k}^{(-i_1,\cdots,-i_{r},a_k)}$$

\end{lem}
\begin{proof} It suffices to consider the system of $(q-1)^{r}$-linear equations, with variables $X_{ i_1\cdots i_{r} }$
$$\chi_{q}^{a_k}( \alpha_{k1} \cdot  z_1+\cdots +\alpha_{kr} \cdot z_{r}+ \alpha_k)=$$
   $$\sum_{0\leq i_1,\cdots,i_{r} \leq q-2} X_{ i_1\cdots i_{r} }\cdot \chi_{q}^{ i_1}(z_1)\cdots \chi_{q}^{ i_{r}}(z_{r}), $$
where $(z_1,\cdots,z_{r}) \in ({{\Bbb F}_q^\times})^{r} $,
and the equality
$$ J_{\omega_k}^{(-i_1,\cdots,-i_{r},a_k)}=$$
$$\sum_{ (z_1,\cdots,z_{r}) \in {\Bbb F}_q^{r}     }\chi_{q}^{-i_1}(z_1)\cdots \chi_{q}^{-i_{r}}(z_{r})\chi_{q}^{a_k}(\alpha_{k1} \cdot  z_1+\cdots + \alpha_{kr} \cdot z_{r}+\alpha_k ).$$

\end{proof}

\begin{lem}\label{prod} With the above notations, we have that
$$(q-1)^{r(d-r-1)} J_{H^d_r}^{(a)}=$$ $$ \sum [J^{(i_1^{r+2}+\cdots+i^{d }_1+a_1,\cdots, i_r^{r+2}+\cdots+i^{d}_r+a_r,a_{r+1})}_{\omega_{r+1}}\cdot \prod _{k={r+2}} ^dJ^{(-i_1^{k},\cdots,-i_r^{k}, a_{k})}_{\omega_{k}}] $$
where the sum is over all $0\leq i_s^k \leq q-2$, with $1\leq s \leq r$.
\end{lem}
\begin{proof} From the above Lemma we have that
$$ \chi_{q}^{a_d}( \alpha_{d1}\cdot  z_1+\cdots + \alpha_{dr } \cdot z_{r }+ \alpha_d )=$$
$$  (q-1)^{-r}  \sum_{0\leq i^d_1,\cdots,i^d_{r} \leq q-2}  \chi_{q}^{ i^d_1}(z_1)\cdots \chi_{q}^{ i^d_{r}}(z_{r})\cdot J_{\omega_k}^{(-i^d_1,\cdots,-i^d_{r},a_d)}.$$

With using this equality,   by replacing $z_d$ by $\alpha_{d1}\cdot  z_1+\cdots + \alpha_{dr } \cdot z_{r }+ \alpha_d$ in
$(q-1)^{r}J_{H^d_r}^{(a)}:=
 (q-1)^{r}\underset { z\in H^d_r     }\sum \chi_q^{a_1}(z_1)\cdots \chi_q^{a_d}(z_d)$, we obtain
$$\sum_{z_1,\cdots, z_{d-1} }\chi_{q}^{a_1+i_1^d}(z_1) \cdots \chi_{q}^{a_r+i_r^d}(z_r) \cdot \chi_{q}^{a_{r+2}}(z_{r+1}) \cdots \chi_{q}^{a_{d-1} }(z_{d-1})\cdot J^{(-i_1^{d},-\cdots,-i_r^{d}, a_{d})}_{\omega_{d}}.$$
 The sum is over the affine space within ${\Bbb F}_q^{d-1}$ defined by
$$\{(z_1,\cdots,z_{d-1})\vert \omega_{r+1}(z_1,\cdots, z_{d-1})=0,\cdots, \omega_{d-1}(z_1,\cdots, z_{d-1})=0\}.$$

We prove the Lemma by repeating this process with $z_{d-1}, \cdots , z_{r+2}$ and by considering, in the last step,    the Jacobi sums $$J^{(i_1^{r+2}+\cdots+i^{d }_1+a_1,\cdots, i_r^{r+2}+\cdots+i^{d}_r+a_r,a_{r+1})}_{\omega_{r+1}}.$$

\end{proof}

\begin{rem} Note that for $r=d-1$, we have the formula $J_{H^d_{d-1}}^{(a)}=  J^{(a_1,\cdots ,a_{d})}_{\omega_{d}}$.
\end{rem}

 We denote by $Q:O_{E}={\Bbb F}_q\times O_{E'}\to   O_{E'}$ ($k(\infty)\simeq {\Bbb F}_q$)  the projection morphism. Note that $1\times u \in H_\mathfrak m^{D+i\cdot D_1}\subset {\Bbb F}_q\times O_{E'}$ if and only if $u\in \overline H_\mathfrak m^{D+i\cdot D_1}:=Q[(1\times O_{E'})\cap  H_\mathfrak m^{D+i\cdot D_1}]\subset  O_{E'}\simeq {\Bbb F}_q^d$. In the next Theorem, we follow the notation of Proposition \ref{series}. We denote $r:= h^0(O_X(D+i\cdot D_1-E))$,
where $(a):=(a_1,\cdots,a_d)$,  $a_i:=n_i\frac{q-1}{n}$, $1\leq i\leq d$, and the $n_i$'s are introduced in section 2.
\begin{thm}\label{Jacob} We have
$$L(X,\chi_f,t)=\sum_{[D]}   \chi_f( {\rho( l^D_{\mathfrak m})})   [ \underset{i=0}{\overset
{ 2g+d-1}  \sum}  q^{h^0  ({\mathcal O}_{X  }(D+i\cdot D_1-E))} J^{(a)}_{ \overline H_\mathfrak m^{D+i\cdot D_1}} \cdot t^{ i}].$$
The   sum is over all the classes (with the algebraic equivalence) of divisors of degree $0$, $D$,     on $X\setminus T'$.
Moreover, $(q-1)^{r(d-r-1 )}J^{(a)}_{ \overline H_\mathfrak m^{D+i\cdot D_1}}$   can be expressed as a sum of products of $(d-r)$ Jacobi sums.
\end{thm}
\begin{proof} For the first assertion it suffices to bear in mind   Proposition \ref{series}.   By considering
 the isomorphism $O_{E'}^\times \simeq {\Bbb F}^\times _q\overset d{\times \cdots \times }{\Bbb F}^\times _q$  and that
  $\chi_f(\eta(1,u^{-1}_1\cdots,u^{-1}_d))=\chi^{ a_1  }_{q}(u_1 )\cdots \chi^{a_d }_{q}(u_d )  $ we obtain,
  $$\sum_{u \in  \overline H_\mathfrak m^{D+i\cdot D_1}}\chi_f(\eta(1\times u^{-1} ))=J^{(a)}_{ \overline H_\mathfrak m^{D+i\cdot D_1}}.$$

The last assertions follows  from Lemma \ref{prod}.
\end{proof}
Let ${\Bbb F}$ be an algebraic closure of ${\Bbb F}_q$. We have that
$$({\rm{Pic}}^0_Y({\Bbb F}))^\vee:={\rm{Hom}}( {\rm{Pic}}^0_Y({\Bbb F}), {\Bbb Q}/{\Bbb Z})\simeq \widehat{\Bbb Z}_{(p)}^{2\pi} \times
\widehat {\Bbb Z}_p^\delta,$$
 $\pi$  being the genus of the curve $Y$ and $\delta \in {\Bbb N}$ with $\delta \leq \pi$,
$\widehat{\Bbb Z}_{(p)}$ and $\widehat {\Bbb Z}_p$ denote  the completion of  ${\Bbb Z}$ along the primes of $ {\Bbb Z}\setminus \{p\}$   and  $p$, respectively.
We have the decomposition into eigenspaces
 $$({\rm{Pic}}^0_Y({\Bbb F}))^\vee=\oplus_{0\leq j\leq n-1}({\rm{Pic}}^0_Y({\Bbb F}))^{\chi_f^j}.$$
 Since $n\vert q-1$ we have that the $q$-Frobenius morphism, $ F_q$, is a  $\widehat {\Bbb Z}[\epsilon_n]$-endomorphism in each eigenspace $({\rm{Pic}}^0_Y({\Bbb F}))^{\chi_f^j}$. We consider the endomorphism   $L_q:=F_q-Id$. Its characteristic polynomial satisfies
 $det(t-L_q )=L(X,\chi^j_f,(t+1)^{-1})(t+1)^e$ with $e:=2g+d-1$.

 We define the $k$-fitting ideal for the  endomorphism
 $$L_q:({\rm{Pic}}^0_Y({\Bbb F}))^{\chi_f }\to ({\rm{Pic}}^0_Y({\Bbb F}))^{\chi_f },$$ $I_{k}(F_q),$
 as the annihilator ideal of the ${\Bbb Z}$-module ${\rm Coker}^{\chi_f}  ( \bigwedge^{e-k} F_q)$. Because the coefficient of $t^k$ of
 $det(t-L_q )$ is the trace of $  \bigwedge^{e-k} F_q $,  we have that this coefficient annihilates the ${\Bbb Z}$-module ${\rm Coker}^{\chi_f}  ( \bigwedge^{e-k} F_q)$, and hence from Theorem \ref{Jacob} we obtain:

 \begin{cor} Bearing in mind the above notations, we have that
 $$\sum_{[D]}   \chi_f( {\rho( l^D_{\mathfrak m})})   [ \underset{i=k}{\overset
{ e}  \sum}  q^{h^0  ({\mathcal O}_{X  }(D+i\cdot D_1-E))} J^{(a)}_{ \overline H_\mathfrak m^{D+i\cdot D_1}} \cdot (e-i)^{ [k]}]\in  I_k(F_q),$$

where $(e-i)^{ [k]}:=(e-i)( e-i -1)\cdots ( e-i -k+1)$, $(e-i)^{ [0]}:= 1$ and the   sum is over all the classes (with the algebraic equivalence) of divisors of degree $0$, $D$,     on $X\setminus T'$. Moreover,  $(q-1)^{r(d-r-1 )}J^{(a)}_{ \overline H_\mathfrak m^{D+i\cdot D_1}}$   can be expressed as a sum of products of  Jacobi sums.
\end{cor}

For $X={\Bbb P}^1_q$, the projective line over ${\Bbb F}_q$, and the ideal $\mathfrak m=(x(x-1)(x-\alpha_3)\cdots (x-\alpha_d))\cdot \mathfrak m_\infty$, we denote $\overline H_\mathfrak m^{r\cdot \infty }$ by $\overline H^{r\cdot \infty }$.

\begin{cor}\label{P} When $X={\Bbb P}^1_q$, by following the notation of Example 2,  we have that
$$L({\Bbb P}^1_q,\chi_f,t)= \sum_{r=0}^{d-1}J^{(a)}_{\overline H^{r\cdot \infty }}\cdot t^{ r},$$
where $\overline H^{r\cdot \infty }=\{  v(x)\}\subset  \frac{{\Bbb F}_q[x]}{p(x)}$ and $\{v(x)\}$  is the set of monic polynomials of degree $r $. Moreover, $(q-1)^{r(d-r-1)}\cdot J^{(a)}_{\overline H^{r\cdot \infty }}$    is a sum of  products of  $(d-r)$ Jacobi sums.
\end{cor}
\begin{proof}
By taking $D_1=\infty$, this follows from Proposition \ref{series},     Example  2  and     Theorem \ref{Jacob}.
\end{proof}

In the case of  $J^{(a)}_{\overline H^{1\cdot \infty }}$ we can make a more explicit calculation than that made in Theorem \ref{Jacob}. This calculation will be made in the next section.


\section{Explicit Calculations and Aplications}

Bearing in mind    Corollary \ref{P}, we consider the $L$-series,
$$L({\Bbb P}^1_q,\chi_f,t)= \sum_{r=0}^{d-1}J^{(a)}_{\overline H^{r\cdot \infty }}\cdot t^{ r},$$
where
 $f:=(x-\alpha_1)^{n_1}\cdots (x-\alpha_d)^{n_d}$ with $\alpha_1=0$ and $\alpha_2=1$.
We find a formula for  the coefficients $J^{(a)}_{\overline H^{r\cdot \infty }}$ by expressing them as linear combinations of products of Jacobi sums  $J^{(u_1,\cdots,u_r,b)}$  and terms $\chi_q(\alpha_i-\alpha_j)$. We also find applications of this result.

In the following Theorem we denote
$$C_{i_1^{r+2},i_2^{r+2},\cdots,i^d_r}:=(-1)^{r(d-r)+a_1+\cdots+a_d}  \prod_{1\leq l\leq r} (\alpha_{l}-\alpha_{r+1})^{i_l^{r+2}+\cdots+i_l^d}\prod_{l,j=1, (j\neq l)}^{r+1}(\alpha_j-\alpha_l)^{a_j} \cdot $$
$$\cdot \prod_{k=r+2}^d(-1)^{a_k}[ \prod_{j=1}^r(\alpha_{j}-\alpha_k)^{a_k-i^k_j}]$$

\begin{thm}\label{V} We have that
$$(q-1)^{ r(d-r-1)} J^{(a)}_{\overline H_{\alpha_1,\cdots,\alpha_d}^{r\cdot \infty }}=$$
$$\sum [\chi_q(C_{i_1^{r+2},i_2^{r+2},\cdots,i^d_r})\cdot J^{(i^{r+2}_1+\cdots+i_1^d+a_1,\cdots,i_r^{r+2}+\cdots+i_r^d+a_r,a_{r+1})}\cdot \prod_{k=r+2}^dJ^{(-i_1^k,\cdots,-i_r^k,a_k)}]$$
where the sum is over all  $0\leq i^{r+2}_1,i^{r+2}_2, \cdots, i^d_r\leq q-2$.
\end{thm}
\begin{proof}By using the notation of Corollary \ref{P}, if we have that
$$A_1+A_2 x+ \cdots+A_{r}x^{r-1}+x^r\in  H^{r\cdot \infty }\subseteq
\frac{{\Bbb F}_q[x]}{p(x)}\simeq   \frac{{\Bbb F}_q[x]}{ (x-\alpha_1)}\times\cdots \times \frac{{\Bbb F}_q[x]}{ (x-\alpha_d)}\simeq {\Bbb F}^{d}_q, $$
with $\alpha_1=0,  \alpha_2=1$ then the vectors $(z_1,\cdots,z_d)\in H^{r\cdot \infty }\subseteq {\Bbb F}^{d}_q$ is given by the $d$-linear equations
$$ \begin{matrix}    A_1+A_2\alpha_1+\cdots+A_{r}\alpha_1^{r-1}+\alpha_1^r=z_1
 \\\cdot  \\\cdot \\\cdot \\A_1+A_2\alpha_d+\cdots+A_{r}\alpha_d^{r-1}+\alpha_d^r=z_d  \end{matrix} .$$
From the first $r$ equations we obtain for $1\leq j\leq r$

$$A_j=\frac{ \left| \begin{matrix}    1 & \alpha_1 &\cdots &\overset {j}{z_1-\alpha_1^r} & \cdots &\alpha_1^{r-1}
  \\\cdot & \cdot &\cdots &\cdot &\cdots &\cdot \\\cdot & \cdot &\cdots &\cdot &\cdots &\cdot \\1 & \alpha_r &\cdots &{z_r-\alpha_r^r} & \cdots &\alpha_r^{r-1} \end{matrix}\right| }{V(\alpha_1,\cdots,\alpha_r)}$$
  where $V(\alpha_1,\cdots,\alpha_r)$ denotes the Vandermonde determinant
  $$\left|\begin{matrix}    1 & \alpha_1 &    \cdots &\alpha_1^{r-1}
  \\ \cdot & \cdot   &\cdots &\cdot \\ \cdot & \cdot &\cdots  &\cdot \\1 & \alpha_r    & \cdots &\alpha_r^{r-1} \end{matrix}\right|.$$

We also denote by $ V_j^h$    the Vandermonde determinant $ V (\alpha_1,\cdots,\alpha_r)$  without the $h$-row and $j$-column.
   By developing the determinant of the denominator of $A_j$ by the $j$-column we obtain
  $$A_j=\frac{(-1)^{j+1}z_1V_j^1 +\cdots+(-1)^{j+r}z_{r}V_j^{r} +
  (-1)^{r-j}V(\alpha_1,\cdots,\alpha_r)}{V(\alpha_1,\cdots,\alpha_r)} .$$

  By replacing the values of the $A_j$  in the last $(d-r)$ equations of the above  $d$-linear system of equations, we obtain for $1\leq j \leq d-r$, the equations:
$$z_1V(\alpha_2,\cdots, \alpha_{r},\alpha_{r+j})+z_2V(\alpha_1,\alpha_{3},\cdots \alpha_{r}, \alpha_{r+j})+\cdots+
  z_rV(\alpha_1, \cdots, \alpha_{r-1}, \alpha_{r+j})+$$ $$+z_{r+j}V(\alpha_1, \cdots, \alpha_{r} )+V(\alpha_1, \cdots, \alpha_{r}, \alpha_{r+j})=0.$$
  Bearing in mind that $\frac{V(\alpha_1,\cdots,\overset \wedge {\alpha_{i}}, \cdots, \alpha_{r},\alpha_{r+j})}{V(\alpha_1,\cdots,    \alpha_{r }, \alpha_{r+j})}=-\prod_h  (\alpha_i-\alpha_h)$
  where this product is over all $h\in \{1,\cdots, r,r+j\}$ with $h\neq i$, from the above  equations we deduce the $(d-r)$-linear equations for $1\leq j\leq d-r$.
  $$z_1 \prod_{i_1} (\alpha_1-\alpha_{i_1})^{-1}+\cdots+z_r \prod_{i_r} (\alpha_r-\alpha_{i_r})^{-1}+
  z_{r+j} \prod_{i_{r+j} } (\alpha_{r+j}-\alpha_{i})^{-1}-1=0.$$
These products are defined over all  $i_h  \in \{1,\cdots, r,r+j\}$, with  $i_h\neq h$. We denote these equations by  $\omega_{r+j}(z_1,\cdots,z_d)=0$, for $1\leq j\leq d-r$.

By Lemma \ref{prod} we have
$$(q-1)^{r(d-r-1)} J_{H^{r\cdot \infty }}^{(a)}=$$ $$ \sum [J^{(i_1^{r+2}+\cdots+i^{d }_1+a_1,\cdots, i_r^{r+2}+\cdots+i^{d}_r+a_r,a_{r+1})}_{\omega_{r+1}}\cdot \prod _{k={r+2}} ^dJ^{(-i_1^{k},\cdots,-i_r^{k}, a_{k})}_{\omega_{k}}]. $$

Bearing in mind that $  H^{r\cdot \infty }$ is given by the above $(d-r)$-linear equations $\omega_{r+j}(z_1,\cdots,z_d)=0$, $1\leq j \leq d-r$,
and that if $\omega_{r+j}(z_1,\cdots,z_d)=b_1 z_1+\cdots+b_{r}z_{r }+b_{r+j}z_{r+j}-1$
then
$$J^{(j_1,\cdots,j_r,a)}_{\omega_{r+j}}=$$ $$=(-1)^{r+j_1+\cdots+j_r+a }\chi_q^{-j_1}(b_1)\cdot \cdot \cdot\chi_q^{-j_r}(b_r)\cdot\chi_q^{-a}(b_{r+j})J^{(j_1,\cdots,j_r,a)}.$$

By  direct calculation,  using the last assertions, one concludes the Theorem.
\end{proof}

For $r=1$, $\alpha_1=0$ and $\alpha_2=1$, we have that

$$ H^{1\cdot \infty }\subseteq
\frac{{\Bbb F}_q[x]}{p(x)}\simeq  \frac{{\Bbb F}_q[x]}{x }\times \frac{{\Bbb F}_q[x]}{ (x-1) }\times \frac{{\Bbb F}_q[x]}{ (x-\alpha_3)}\times\cdots \times \frac{{\Bbb F}_q[x]}{ (x-\alpha_d)}\simeq {\Bbb F}^{d}_q$$
is given by the equations
 $$(*)\{ (z_1,z_2,\cdots,z_d)\in {\Bbb F}^{d+1}_q \text{ with }  z_2-z_1=1, z_3-z_1=\alpha_3,\cdots,z_d-z_1=\alpha_d\}.$$

 \begin{lem}\label{a}With the above notations we have that
 $$(q-1)^{d-2}(-1)^{d-1+a_2+\cdots+a_d}J^{(a)}_{ H^{1\cdot \infty }}=$$
 $$=
 \sum_{0\leq i_1,\cdots,i_d \leq q-2}\chi_q^{a_3-i_3}(\alpha_3)\cdots \chi_q^{a_d-i_d}(\alpha_d)\cdot J^{(a_1+i_3+\cdots+i_d,a_2)}(\chi_q)\cdot \prod_{k=3}^dJ^{(-i_k,a_k)}.$$
 \end{lem}
 \begin{proof} This results from Theorem \ref{V} by considering $r=1$, $\alpha_1=0$ and  $\alpha_1=1$. It may be also deduced easily from the equations (*) of $ H^{1\cdot \infty }$.
\end{proof}

\begin{rem}\label{R} We consider the curve $y^n= x^{a_1}(x-1)^{a_2} (x-\alpha)^{a_3}$. We have that for $f=x^{a_1}(x-1)^{a_2} (x-\alpha)^{a_3}$
 $$L({\Bbb P}^1_q,\chi_f,t)=1+c_{1}(\alpha_3,q)t +c_{2}(\alpha_3,q)\cdot t^{2} .$$
 By applying Lemma \ref{a}, we have that

 $$c_1(\alpha, q)=(q-1)^{-1}(-1)^{a_2+a_3}\sum_{i=0}^{q-2}\chi_q(\alpha)^{a_3-i}J^{(a_1+i,a_2)}J^{(-i,a_3)}$$

and, for example, by using \cite{Al} Theorem 1, we have that the constant term, $c_{2}(\alpha_3,q) $, is equal to the Jacobi sum $-J^{(a_1,a_2,a_3)}$ and therefore that
$$L({\Bbb P}^1_q,\chi_f,t)=$$ $$1+[(q-1)^{-1}(-1)^{a_2+a_3}\sum_{i=0}^{q-2}\chi_q(\alpha)^{a_3-i}J^{(a_1+i,a_2)}J^{(-i,a_3)}]\cdot t -J^{(a_1,a_2,a_3)}\cdot t^{2}$$

\end{rem}

\begin{lem} With the above notations if $a_3=a_2$, $2(a_1+a_2)\equiv 0$ mod $(q-1)$ and $a_1+a_2 \neq q-1$ then
when  $\alpha^{\frac{q-1}{2}}=-1$ mod $p$ we have that $ c_1(q,\alpha)=0$.
\end{lem}
\begin{proof}  By considering the formula for Jacobi sums $J^{(b,c)}=J^{(-b-c,c)}$ we have that
$$\sum_{i=0}^{q-2}\chi_q(\alpha)^{a_2-i}J^{(a_1+i,a_2)}J^{(-i,a_2)}=
\sum_{i=0}^{q-2}\chi_q(\alpha)^{a_2-i}J^{(a_1+i,a_2)}
J^{( i-a_2,a_2)}.$$
By considering $j:=i-a_2$, bearing in mind   hypothesis  $2(a_1+a_2)\equiv 0$ mod $(q-1)$ and $a_1+a_2 \neq (q-1)$, we have that this sum is equal to
$$\sum_{j=-a_2}^{q-2-a_2}\chi_q(\alpha)^{-j}J^{(j+\frac{q-1}{2},a_2)}J^{(j,a_2)}=$$
$$=\sum_{j=-a_2}^{\frac{q-3}{2}-a_2}
(\chi_q(\alpha)^{-j}+\chi_q(\alpha)^{-j+\frac{q-1}{2}})J^{(j+\frac{q-1}{2},a_2)}J^{(j,a_2)},$$
and we conclude. Note that $\chi_q(\alpha)^{\frac{q-1}{2}}+1=0$ if and only if  $\alpha^{\frac{q-1}{2}}=-1$ mod $p$
\end{proof}

To obtain applications for the above Lemma, we consider the curve $Y\equiv y^4=x(x-1)(x-\lambda)$ defined over ${\Bbb Z}[i]$ and with genus  $g=3$. Moreover,  the finite field ${\Bbb F}_q$ will be a residual field of characteristic $p$ of this ring. If $\lambda^{\frac{q-1}{2}}=-1$,  then the zeta function of the quotient variety defined over ${\Bbb F}_q$,
 ${\rm{Pic}}^0(Y)/E$, is equal to
 $$\frac{(q^{-2s}-J^{(\frac{q-1}{4},\frac{q-1}{4},\frac{q-1}{4})})\cdot  (q^{-2s}-J^{(-\frac{q-1}{4},-\frac{q-1}{4},-\frac{q-1}{4})})}{(1-q^{-s})(1-q^{-s+1})}$$
 $E$ being the elliptic curve $y^2=x(x-1)(x-\lambda)$.

We consider another application of Lemma \ref{a}.  For the elliptic curve  $y^2=x(x-1)(x-\lambda)$  we have
$$L({\Bbb P}^1_q,\chi_f,t)=1+c_{1}(\lambda,q)t +q \cdot t^2.$$

From Remark \ref{R}, we obtain 
$$c_1(\lambda, q)=(q-1)^{-1} \sum_{i=0}^{q-2}\chi_q(\lambda)^{\frac{q-1}{2}-i}J^{(\frac{q-1}{2}+i,\frac{q-1}{2})}J^{(-i,\frac{q-1}{2})}.$$
 By using  the formula $J^{(b,c)}=J^{(-b-c,c)}$
  again, we obtain
 $$c_1(\lambda, q)=(q-1)^{-1}\sum_{i=0}^{q-2}\chi_q(\lambda)^{\frac{q-1}{2}-i}(J^{(\frac{q-1}{2}+i,\frac{q-1}{2})})^2.$$

By considering $a:=\frac{q-1}{2}$ and $j:\frac{q-1}{2}-i$ we have
$$(**)\quad c_1( \alpha,q)=(q-1)^{-1}\sum_{j=-\frac{q-3}{2}}^{\frac{q-1}{2}}\chi_q(\lambda)^{j}(J^{(-j,\frac{q-1}{2} )})^2 .$$

In this way, we have that $\prod_{\lambda\in {\Bbb F}^\times_q  } (x-c_1(  \lambda,q)) $ is the characteristic polynomial of the Hemitian matrix

$$(q-1)^{-1}\left(\begin{matrix}  (J^{(1-q,a)})^2 & (J^{(-1,a)})^2  &    \cdots &(J^{(2-q,a)})^2 \\  (J^{(2-q,a)})^2&(J^{(1-q,a)})^2 &\cdots & (J^{(3-q,a)})^2 \\ \cdot & \cdot &\cdots  &\cdot \\(J^{(-1,a)})^2 & (J^{(-2,a)})^2   & \cdots &(J^{(1-q,a)})^2 \end{matrix}\right).$$
Note that this matrix is Hermitian because the conjugate of the Jacobi sum $J^{(b,c)}$ is equal to $J^{(-b,-c)}$.

The Hermitian forms associated with the above matrix and
$$ \left(\begin{matrix}  c_1(1, q) & 0  &    \cdots &0 \\  0&c_1(2, q)&\cdots & 0 \\ \cdot & \cdot  &\cdots  &\cdot \\0 & 0  & \cdots &c_1(q-1, q)\end{matrix}\right)$$
are equivalent. Therefore, the signature of the form associate with the above
Hermitian matrix gives us the amount of   $\lambda \in {\Bbb F}_q^\times$, such that the cardinality of the rational number of $y^2=x(x-1)(x-\lambda)$ over ${\Bbb F}_q $ is either greater or less than $q+1$. The rank of this matrix also gives the amount of  $\lambda \in {\Bbb F}^\times_q $ for which the above curves are supersingular.

Thus, if we denote by $M_j$ the $(j\times j)$-minor of the above Hermitian matrix,
$$\left|\begin{matrix}  (J^{(1-q,a)})^2 &      \cdots &(J^{(1-j,a)})^2 \\    \cdot &  \cdots  &\cdot \\(J^{(j-q,a)})^2 &      \cdots &(J^{(1-q,a)})^2 \end{matrix}\right|$$
then the cardinal of the set of positive numbers of $$\{c_1(q,1),c_1(q,2), \cdots, c_1(q,q-1)\} \text{ and } \{M_1,\cdots,M_{q-1}\}$$ coincides.

The following Lemma makes the formula,  module $p$, of  Theorem \ref{V} easier. Let us consider $i_1,\cdots,i_r, b\in {\Bbb Z}$, with  $ \vert i_1\vert,\cdots,\vert i_r  \vert \leq q-2$, and $0 \leq b\leq q-2$.

\begin{lem} We have that $J^{(-i_1,\cdots,-i_r,b)}\equiv \left(\begin{matrix}    b
 \\ i_1 \cdots i_r  \end{matrix}\right)$ mod $p$ if either $i_1,\cdots,i_r\leq b$ or $(q-1)+i_1,\cdots,(q-1)+i_r\leq b$ and $J^{(-i_1,\cdots,-i_r,b)}\equiv 0$ mod $p$ in the other cases.
 Here, we denote $\left(\begin{matrix}    b
 \\ i_1 \cdots i_r  \end{matrix}\right):=\frac{b !}{i_1! \cdots i_r!}$.
 \end{lem}
 \begin{proof} It suffices bear in mind  Lemma \ref{car} and that
 $\chi^b_q(z_1+\cdots+z_r+1)\equiv (z_1+\cdots+z_r+1)^b\text{ mod } p$.
 \end{proof}

 We have that, for example  the elliptic curve, $y^2=x(x-1)(x-\lambda)$,  by using the formula(**) the term  $c_1(q,\lambda)$  mod $p$ gives the Deuring polynomial
$$\sum^{\frac{q-1}{2}}_{j=0}\lambda^{j}
 \left(\begin{matrix}  \frac{q-1}{2} \\ j  \end{matrix} \right)^2.$$

 In general, if we denote
  $$L({\Bbb P}^1_q,\chi_f,t)=1+c_1(\alpha_3,\cdots,\alpha_d,q)\cdot t+\cdots +c_{d-1}(\alpha_3,\cdots,\alpha_d,q)\cdot t^{d-1} $$
  then by using the above Lemma and Theorem \ref{V} we have that, mod $p$,
   $c_r(\alpha_3,\cdots,\alpha_d,q)$, with $1\leq r\leq d-1$, is equal to
 $$\sum_{i_1^{r+2},i_2^{r+2},\cdots,i^d_r} C_{i_1^{r+2},i_2^{r+2},\cdots,i^d_r}
  \left(\begin{matrix}    a_{r+1}
 \\ I_1 \cdots I_r   \end{matrix}\right)\prod^d_{k=r+2}\left(\begin{matrix}   a_{k}
 \\ i^k_1 \cdots i^k_r  \end{matrix}\right).$$

 Where  $ (q-1) - I_l$ is the rest mod $(q-1)$ of $i_l^{r+2}+i_l^{r+2}+\cdots+i^d_l+a_l$. The sum is over all $0<i_l^k\leq a_l$. Note that   $\chi_q(\beta)^h= \beta^h$ mod $p$ and that if for some $1\leq l \leq r$ we have $q-1-I_l>a_{r+1}$ then  $\left(\begin{matrix}    a_{r+1}
 \\ I_1 \cdots I_r   \end{matrix}\right)=0$.


\vskip1.5truecm { \'Alvarez V\'azquez, Arturo}\newline {\it
e-mail: } aalvarez@usal.es

\end{document}